%% file: QCarticle.tex
\documentclass[10pt]{amsart}
\usepackage{amsmath}
\usepackage{amsfonts}
\usepackage{amssymb}
\usepackage{graphicx}
\usepackage{float}
\usepackage{xcolor}
\usepackage{hyperref}
\usepackage{mathtools}
\usepackage{nicefrac}
\usepackage{listings}
\usepackage{sty}
\usepackage{mathpazo}
\usepackage{setspace}

\begin{document}
\title{Fast Navigation with Icosahedral Golden Gates} 
\author{Terrence Richard Blackman}
\address[Terrence Blackman]{MEC \& IAS }
\email{tblackman@ias.edu}
\urladdr{http://terrenceblackman.com}
\author{Zachary Stier}
\address[Zachary Stier]{UC Berkeley}
\email{zstier@berkeley.edu}
\date{\today}

\begin{abstract}
An algorithm of Ross and Selinger for the factorization of diagonal elements of PU(2) to within distance $\varepsilon$ was adapted by Parzanchevski and Sarnak into an efficient probabilistic algorithm for any element of PU(2) using at most effective $3\log_p\frac{1}{\varepsilon^{3}}$ factors from certain well-chosen sets associated to a number field and a prime $p$. The icosahedral super golden gates are one such set associated to $\mathbb{Q}(\sqrt{5})$. We leverage recent work of Carvalho Pinto, Petit, and Stier to reduce this bound to $\frac{7}{3}\log_{59}\frac{1}{\varepsilon^3}$, and we implement the algorithm in Python. This represents an improvement by a multiplicative factor of $\log_259\approx5.9$ over the analogous result for the Clifford+$T$ gates. This is of interest because the icosahedral gates have shortest factorization lengths among all super golden gates. 
\end{abstract}

\maketitle

\section{Introduction}
\label{sec:Intro}
\input{Introduction.tex}

\section{Super Golden Gates}
\label{sec:SuperGoldenGates}
\input{SuperGoldenGates}

\iffalse
\section{Nearby Elements in PU(2)}
\label{sec:NearbyElements}
\input{nearbyelementspu2}
\fi

\section{Nearby Elements in PU(2), and Factoring in $\Gamma$}
\label{sec:FactoringGamma}
\input{FactoringGamma}

\section{Algorithm for Short Paths to Diagonal Elements}
\label{sec:Diagonal}
\input{Diagonal}

\section{Algorithm For Short Paths}
\label{sec:AlgShortPath}
\input{AlgForShortPaths}

\section{Implementation Details and Examples}
\label{sec:Ex}
\input{Ex}

\section{Concluding Remarks}
\label{sec:Conclusion}
\input{ConcludingRemarks}

\section{Acknowledgements}
\label{sec:Ack}
\input{Acknowledgements}

\bibliography{bib}{}
\bibliographystyle{plain}

\appendix

\section{$\Q(i,\vf)$ is norm-Euclidean}
\label{sec:norm-Euclidean}
\input{norm-Euclidean}

\section{Irreducible elements and sums of two squares in $\Z[\vf]$}
\label{sec:Zphi appendix}
\input{Zphi}

\end{document}

%% file: Introduction.tex
Lubotzky, Phillips, and Sarnak, in a series of papers \cite{lps1986explicit,lps1986hecke,lubotzky1987hecke,lubotzky1988ramanujan},  explicitly constructed topological generators with optimal covering properties for the compact Lie group $\PU(2)$. 
Such generators for projective unitary groups find an interesting application in quantum computing where a fundamental design challenge is to determine an optimal, fault-tolerant decomposition of a quantum gate.
For classical computing a single bit state is an element of $\{0,1\}$. 
A classical gate implements functions on binary inputs. The only nontrivial single bit logic operation is \texttt{NOT}, which takes 0 to 1 and 1 to 0 (though it is also possible for the codomain to contain more than one bit). 
In the quantum setting, single {\em qubit} states are points 
\begin{equation*}
 u = (u_{1}, u_{2}) \in \mathbb{C}^{2}  
\end{equation*}
up to a mutual phase $e^{2\pi i\gt}$ in each component, such that 
\begin{equation*}
\abs{u}^2 = \abs{u_1}^2 + \abs{u_2}^2 = 1.
\end{equation*}
A gate here {\em cannot} output more than one qubit, and thus must be a $2 \times 2$ projective unitary. 

A \emph{universal gate set} is a finite set of gates, $S:=\{ s_1,s_2,\cdots, s_k : s_\ell \in \PU(2)\}$, that can approximate, in the bi-invariant metric on the compact Lie group, any matrix arbitrarily well.
That is, the group generated by $S$ must be topologically dense in $\PU(2)$. 
The Solovay--Kitaev theorem \cite{nielsen2010quantum} guarantees that universal gate sets can efficiently approximate quantum operations for unitaries on a constant number of qubits.
Universal gate sets typically consists of a finite group $C\leqslant\PU(2)$ together with an extra element $\tau$, which we will take to be an involution, so that the subgroup generated by $C$ and $\tau$ covers $\PU(2)$ with minimal $\tau$-count, and simultaneously navigates $\PU(2)$ efficiently. That is, given some gate in $\PU(2)$ and desired precision $\ve$, there is an efficient algorithm (polynomial-time in the input size) that with high probability finds a short word in $S$ to that precision, typically of length $O(\log(\nicefrac{1}{\ve}))$. 
The deep insight of Sarnak \cite{sarnakletter} is that the construction and optimality of universal single-qubit quantum gate sets can be understood in terms of the arithmetic of quaternion algebras.
Specifically, identifying $C$ with a subgroup of the group of units in a definite quaternion algebra over a totally real number field provides a coherent framework within which one can systematically address the question of optimality of topological generators for $\PU(2).$

The question we address within this context is  that of finding the ``best'' topological generators of $\PU(2)$ among those universal gate sets, the {\em super golden gates} of Parzanchevski and Sarnak \cite{ParzanchevskiSarnak}, which are known to possess optimal covering properties and efficient navigation.
In particular, there are only finitely many super golden gates \cite[p.870--871]{ParzanchevskiSarnak} and the respective finite subgroups $C$ can be realized as the group of symmetries of of the Platonic solids: for the tetrahedron, $A_{4}$; for the cube and the octahedron, $S_{4}$; and for the dodecahedron and icosahedron, $A_{5}$.

We demonstrate that the \emph{icosahedral} super golden gates admit a factorization directly analogous to the one obtained by Stier \cite{Stier}, and that this gives the best-known preconstant to the first-power logarithm in the approximation length. 
The exact factor of the improvement is $\log_259\approx5.9$. 
This improvement is due to the fact that the icosahedral super golden gates have a growth rate that is on the order of $59^k$, while gate set studied in \cite{Stier}, the Clifford+$T$ gates, have growth rate of order $2^k$. 
We note also that the \emph{icosahedral} super-golden-gates represent the greatest number of distinct gates with bounded $\tau$-count. 

The commonly used Clifford+$T$ gate set provides a set of elementary gates that is universal and consists only of a small number of gates, all of which are very well compatible with many established error correction schemes and can be physically implemented in all quantum technologies that seem promising for large-scale quantum computations \cite{niemann2020advanced}.
In this case the finite group $C$ is the Clifford group $C_{24}$ of order 24 in $\PU(2).$
At least one non-Clifford gate must be added to the basic gate set in order to achieve universality. 
A common choice for this additional gate is the $T$-gate (or $\frac{\pi}{8}$-gate). 
The $T$-gate is not the only possible extension of the Clifford group but it is considered to be the most practical one. 
This is due to the availability of fault-tolerant implementations of the $T$-gate. 
For this reason, the Clifford+$T$ gate set is considered the most promising candidate for practical quantum computing.
We recall, for the convenience of the reader, some of the salient details of the single qubit Clifford+$T$ gate set.
Let $$H:=\frac{i}{\sqrt{2}}\begin{pmatrix}1&\phantom{-}1\\1&-1\end{pmatrix} \text{ and } T:=\begin{pmatrix}e^{\nicefrac{i\pi}{8}}&\\&e^{\nicefrac{-i\pi}{8}}\end{pmatrix}.$$ 
We take $S:=\{H, T\}$ \cite{sarnakletter,RossSelinger,Stier}. The Clifford+$T$ universal gate set is an example of a golden gate set \cite{sarnakletter,ParzanchevskiSarnak}. 
The remarkable observation of \cite{sarnakletter} is that the $H$ and $T$ gates of the Clifford+$T$ gate set come from the definite quaternion algebra 
$$\cA:=\left(\frac{-1,-1}{\Q(\sqrt{2})}\right).$$
Kliuchnikov, Maslov, and Mosca \cite{KMM} characterized the group $\<S\>$ and demonstrated an efficient algorithm to factor exactly, by carefully studying the powers of $\sqrt{2}$ in the denominators of matrix entries. 
Ross and Selinger \cite{RossSelinger} then focused on diagonal matrices near to $\<S\>$, with the goal of finding $\ve$-close factorizations of length $c\log_2(\nicefrac{1}{\ve^3})$; the base of 2 is intrinsic to the structure of the Clifford+$T$ gates.
By studying {\em upright sets} in the plane, which function analogously to the simplices of Lenstra's algorithm \cite{Lenstra,Paz}, they achieved the leading coefficient of $1+o(1)$ (in that restricted case). 
Parzenchevski and Sarnak \cite{ParzanchevskiSarnak} generalized Ross and Selinger's approach to any golden gate set (with the base of the logarithm correspondingly changing per the gate set), and by considering Euler angles reached the coefficient of $3+o(1)$ for approximations to generic elements. 
Working with the (finite) LPS Ramanujan graphs (see \secref{sec:ram} and \cite{lubotzky1988ramanujan}), Carvalho Pinto and Petit \cite{CP} factorize in the equivalent of $\nicefrac{7}{3}+o(1).$
This approach is adapted by Stier \cite{Stier} for the same coefficient with Clifford+$T$ gates.
We combine aspects of the techniques of \cite{CP,Stier}, in the proposed icosahedral setting, to reduce this bound to $\frac{7}{3}\log_{59}(\nicefrac{1}{\ve^3}).$

%% file: SuperGoldenGates.tex
We recall here essential ideas related to the arithmetic of quaternion algebras \cite{Vigneras,voight}, $S$-arithmetic groups \cite[\S C]{morris2001introduction}, and Ramanujan graphs \cite{lubotzky2020ramanujan} which lead to the icosahedral super golden gates \cite{ParzanchevskiSarnak}. 
\subsection{Quaternion Algebras}
A quaternion algebra $\cA$ over a field $\F$ is a central simple algebra of dimension four over $\F$ (cf. \cite[\S5.2]{Miyake} or \cite[p.15]{Stefan1}).
It follows from Wedderburn's structure theorem on simple algebras \cite{Dubisch} that every quaternion algebra over any field $\F$ is either isomorphic to $\M_2(\F)$ or a division algebra with center $\F$ \cite{Lewis}.
If the characteristic of $\F$ is not 2 then it is always possible to find a basis $\{1, i ,j ,k\}$ for $\cA$ over $\F$ such that
$$i^{2}=\alpha, j^{2}=\beta, k=ij=-ji$$
where $\alpha,\beta\in\F^\times$.  
We designate such an algebra by 
$$\cA=\left(\frac{\alpha,\beta}{\F}\right).$$
Evidently $q\in\cA$ is of the form
\begin{equation}
q=x_{0}+x_{1}i+x_{2}j+x_{3}k,\label{eq:gen form alg elt}
\end{equation} 
where $x_\ell\in\F$. 
In this notation, Hamilton's quaternions arise as
$$\h = \left(\frac{-1,-1}{\R}\right).$$

The conjugate of $q$ as in \eqref{eq:gen form alg elt}, denoted by $\bar{q}$, is equal to $x_{0}-x_{1}i-x_{2}j-x_{3}k$. 
For each $q\in\cA$ we define the (reduced) norm map $\nm:\cA\to\F$ by 
$$\nm(q):=q\bar{q}=x_{0}^{2}-\alpha x_{1}^{2}-\beta x_{2}^{2}+\alpha\beta x_{3}^{2}$$
and the (reduced) trace map $\tr:\cA\to\F$ by 
$$\tr(q):=q+\bar{q}=2x_{0}.$$
We note that every element $q\in\cA$ satisfies the quadratic equation 
$$q^{2}-\tr(q)q+\nm(q)=0.$$
It is possible to make everything completely explicit by embedding $\cA$ in $\M_2\lpr{F\lpr{\sqrt{\alpha}}}$ by, for
example, 
$$i\mapsto\begin{pmatrix}\sqrt{\alpha}\\&-\sqrt{\alpha}\end{pmatrix}\text{ and }
j\mapsto\begin{pmatrix}&\beta\\1\end{pmatrix}.$$
Evidently, if $\alpha$ is a square in $\F$, then
$\cA\cong\M_2(\F)$. 
A necessary and sufficient condition for
$\cA\cong\M_2(\F)$ is that $\alpha$ is the norm of an element in
$\F(\sqrt{\beta})$ with respect to $\F$
\cite{voight}. 
Of course, one may interchange $\alpha$ and $\beta$ in this remark.

Specializing to quaternion algebras over the rational field $\Q$, or over one of the completions $\Q_{p}$ or $\R$ (the completion $\Q_\infty$), 
let $\cA$ be a quaternion algebra over $\Q$ and let $p$ be a prime (or $\infty$). 
We define 
$$\cA_p := \cA\underset{\Q}{\otimes}{\Q_p }.$$
Either $\cA_{p} \cong \M_2(\Q_p)$ or $\cA_{p}$ is a division algebra over $\Q_p$. 
In this case we will have that $\cA_p\cong\h_p$ and say that $\cA$ is {\em ramified} at $p.$ 
When $\cA_p \cong \M_2(\Q_p)$ we will say that $\cA$ is {\em unramified} or {\em split} at $p.$ 
If $\cA$ is ramified at $\infty$ it is called a {\em definite rational quaternion algebra}---that is, if $\cA$ is definite then
$$\cA_\infty = \cA\underset{\Q}{\otimes}{\R}\cong\h.$$
If $\cA$ is split at $\infty$, it is called an {\em indefinite rational quaternion algebra}---that is, if $\cA$ is indefinite then
$$\cA_{\infty} = \cA\underset{\Q}{\otimes}{\R}\cong \M_2(\R).$$

\subsubsection{Definite quaternion algebras over totally real number fields}
We now turn to the quaternion algebras that are the objects of interest in this paper, definite quaternion algebras over totally real number fields.
Let $\cA$ be a quaternion algebra over a number field $\F,$ $\nu$ be a place of $\F$, and $\F_\nu$ be the completion of $\F$ at $\nu$. 
Recall that a \emph{totally real number field} is a finite algebraic extension of $\Q$ all of whose complex embeddings lie entirely in $\R.$ 
For example, the field $\Q(\sqrt{d})$ is totally real for positive, integral $d.$
If every infinite place of $F$ is ramified in $\cA,$ we say that $\cA$ is a totally definite quaternion algebra. 
Consequently, if $\cA$ is a totally definite quaternion algebra over a number field $\F,$ then $\F$ is necessarily totally real. 
Moreover, if $\F$ is a quadratic field, the number of finite places which are ramified in $\cA$ is even.

\subsubsection{Unit groups in orders in definite quaternion algebras over totally real number fields}
Let $\cA$ be quaternion algebra over a field $\F$ and let $O_\F$ be the ring of integers in $\F$. 
An element $q\in\cA$ is {\em integral} if $\nm(q),\tr(q)\in O_\F$. 
An {\em order} $\cO\subseteq\cA$ is a $O_\F$-algebra of integral elements such that \cite[p.2]{andreas1}
$$\cO\underset{O_\F}{\otimes}{\F}\cong\cA.$$ 
Observe that as an example of an order we always have 
\begin{equation}
    \cO:=O_\F\oplus O_\F i\oplus O_\F j\oplus O_\F k.
\end{equation}
If $\cO$ is an $O_\F$-order in a definite quaternion algebra $\cA$ over the totally real field $\F$ then the group of units of reduced norm 1, i.e.
\begin{equation}
\cO^1:=\{\alpha\in\cO : \nm(\alpha)=1\}
\end{equation} 
is a finite group \cite[p.289]{voight}.
These finite groups will correspond to the groups of rotational symmetries of the platonic solids \cite[p.172--173]{voight}.

\subsubsection{Ramanujan graphs}\label{sec:ram}
Two key ideas are needed for the construction of golden gate and super golden gate sets: a $S$-arithmetic unit quaternion group and a Ramanujan graph.
We outline the essential ideas of Lubotzky, Phillips, and Sarnak's \cite{lubotzky1988ramanujan} ``LPS'' construction of these objects.

Ramanujan graphs are graphs whose spectrum is bounded optimally. 
Let $X$ be a finite connected $k$-regular graph and $A$ its adjacency matrix.
\begin{defn}
The graph $X$ is called a Ramanujan graph if every eigenvalue $\lambda$ of $A$ satisfies either $\lvert \lambda \rvert = k$ or $\lvert \lambda \rvert \le 2\sqrt{k-1}.$ 
\end{defn}
LPS Ramanujan graphs \cite{davidoff2003elementary} arise as Cayley graphs of $\PSL(2,\F_q)$ (for $\F_q$ the finite field on $q$ elements).
When considering these graphs we interchangeably refer to their elements by their group-theoretic properties as matrices and their graph-theoretic relations as vertices.
\cite{lubotzky1988ramanujan} establishes that for any prime $p \equiv 1 \pmod{4}$ there are infinitely many ($p+1$)-regular Ramanujan graphs.
We use the notation $X^{p, q}$ where $p$ and $q$ are distinct primes congruent to 1 modulo 4 to represent such graphs.
The construction comes from number theory by way of the generalized Ramanujan conjecture \cite[p.873]{ParzanchevskiSarnak}. 
The symmetric space
$\PGL(2, \Q_{p})/\PGL(2, \Z_{p})$ can be identified with a ($p+1$)-regular infinite tree. 
$\PGL(2,\Z[\nicefrac{1}{p}])$ acts from the left on $\PGL(2, \Q_{p})/\PGL(2, \Z_{p})$. The generalized Ramanujan conjecture, a theorem in this case, implies that the quotient of $\PGL(2, \Q_{p})/\PGL(2, \Z_{p})$ by any congruence subgroup of $\PGL(2,\Z[\nicefrac{1}{p}])$, a ($p + 1$)-regular graph, is a Ramanujan graph. 
By considering an appropriate congruence subgroup of $\PGL(2,\Z[\nicefrac{1}{p}])$ we can identify the quotient of this symmetric space with a Cayley graph associated to $\PSL(2, \F_q)$ or $\PGL(2, \F_q)$, depending on the value of the Legendre symbol $\leg{p}{q}$ \cite{sardari2015diameter}.

\subsubsection{$p$-arithmetic unit quaternion groups}
Golden gate and super golden gate sets for $\PU(2)$ require the construction of a $p$-arithmetic group (a special case of $S$-arithmetic groups, where $S$ is a collection of places of $\Z$). Let $G\le\GL(n)$ be an algebraic group defined over $\Z[\nicefrac{1}{p}]$ with $G(\R)$ compact.
A $p$-arithmetic group $\Lambda$ is a subgroup of $G(\Z[\nicefrac{1}{p}])\leq G(\R)\times G(\Q_p)$, and has {\em congruence subgroup} $\Lambda(N):=\{g\in\Lambda : g\equiv I \pmod{N}\}$. That is, in our compact Lie group $\PU(2)$ we take only rational numbers whose denominators are powers of a fixed prime $p$ as coefficients in the matrices. 

One can also make the same construction over the ring of integers $O_\F$ of a totally real number field $\F\supsetneq\Q$, at which point the role of $p$ is played by $\fp$ a prime ideal of $O_\F$, so that the inverted elements (those allowed in denominators) are now all of $\fp\cap\F^\times$. 

\subsection{Golden gates}
Golden gates are special unit groups in quaternion algebras over totally real number fields derived from the $p$-arithmetic groups \cite{sarnakletter}.
They give variants of optimal generators for $\PU(2)$ and connect the arithmetic of quaternion algebras to quantum computation on a single qubit.
The ``golden'' characterization is to be understood by way of an interesting link between $p$-arithmetic unit groups coming from unit quaternion groups and the Ramanujan graphs $X^{p, q}$, which we explicate below.

Recall once more that in classical computation, one decomposes any function into basic logical gates such as \texttt{XOR}, \texttt{AND}, and \texttt{NOR}, and that in quantum computation, the classical bits are replaced by qubits, which are vectors in projective Hilbert space $\C P^{n}$, and that the logical gates are \emph{all }the elements of the projective unitary group $\PU(2).$ 
Let $S$ be a subgroup of $\PU(2)$ and denote by $S^{(\ell)}$ the set of $\ell$-fold products of elements in $S$. 
If $\left\langle S\right\rangle =\bigcup\limits_{\ell\geq0}S^{(\ell)}$
is dense in $\PU(2)$ (with respect to the standard bi-invariant metric $d^{2}(A,B):=1-\frac{|\tr(A^{*}B)|}{2}$) then $S$ is universal. 
That is, any gate can be approximated with arbitrary precision as a product of elements of $S$.

The notion of of a golden gate set is much stronger, requiring \cite{lubotzky2020ramanujan}:
\begin{enumerate}
    \item {\em Optimal covering of $\PU(2)$ by $\<S\>$}: for every $\ell$ the set $S^{(\ell)}$ distributes
in $\PU(2)$ as a perfect sphere packing (or randomly placed points) would,
up to a negligible factor.
\item {\em Approximation}: given $A\in \PU(2)$ and $\varepsilon>0$, there is an efficient algorithm to find some $A'\in B_{\varepsilon}(A)$ (the $\varepsilon$-ball around $A$) such that $A'\in S^{(\ell)}$ with
$\ell$ (almost) minimal.
\item {\em Compiling}: given $A\in\left\langle S\right\rangle $ as a matrix,
there is an efficient algorithm to write $A$ as a word in $S$ of
the smallest possible length.
\end{enumerate}
These requirements ensure that any gate can be approximated and compiled
as an efficient circuit using the gates in $S$.

\subsection{Super golden gates}
Each super golden gate set is composed of a finite group $C$ and
an involution $\tau$, which lie in a $\fp$-arithmetic
group for $\fp$ a prime ideal of the integers $O_\F$ of
a totally definite quaternion algebra $\cA$, over a totally real number
field $\F$. 
We require that: 
\begin{enumerate}
\item $C$ acts simply transitively on the neighbors of any given vertex in $X^p$, the ($p+1$)-regular tree, for $p$ the norm of $\fp$.
\item $\tau$ is an involution which takes a vertex to one of its neighbors.
\end{enumerate}

$\fp$-arithmetic unit quaternion group act transitively on the vertices of the corresponding $X^p$. 
The $\fp$-arithmetic groups which act transitively on the vertices of $X^p$ are called the {\em golden gates} and the $\fp$-arithmetic groups which act transitively on both the vertices and edges of $X^p$ are called the {\em super golden gates}. 

\section{Icosahedral Super Golden Gates}

Recall that there are only finitely many such super golden gate sets and in each case the finite group $C$ is identified with the group of rotational symmetries of a platonic solid: for the tetrahedron, $A_{4}$; for the cube and the octahedron, $S_{4}$; and for the dodecahedron and icosahedron, $A_{5}$.
Each of these finite groups can be precisely identified with a $\fp$-arithmetic unit quaternion group coming from one of the following quaternion algebras \cite{sarnakletter}:
	\begin{itemize}
		\item tetrahedral gates: $ \cA = \left(\frac{-1,-1}{\Q}\right)$; 
		\item octahedral gates: $ \cA = \left(\frac{-1,-1}{\Q(\sqrt{2})}\right)$; and 
		\item icosahedral gates: $ \cA = \left(\frac{-1,-1}{\Q(\sqrt{5})}\right)$.
	\end{itemize}
We consider the quaternion algebra $\cA$ over the {\em golden field} $\Q(\sqrt{5})$ \cite[p.895]{ParzanchevskiSarnak}.
A maximal order $\cO$ in $\cA:=\left(\frac{-1,-1}{\Q(\sqrt{5})}\right)$ is given by
the ring of icosians.
The unit group $\cO^{1}$ is the platonic icosahedral group,  generated by $\left\langle\frac{1+i+j+k}{2}, \frac{1+\varphi^{-1}j+\varphi k}{2}\right\rangle$, where $\vf:=\frac{1+\sqrt{5}}{2}$ is the golden ratio, and $\cO^{1}\cong A_{5}$
In $\PU(2)$, this corresponds to 
$$C_{60}=\<\begin{pmatrix} 1 &1\\i&-i\end{pmatrix}, \begin{pmatrix} 1& \varphi-\nicefrac{i}{\varphi}\\ \varphi + \nicefrac{i}{\varphi}& -1\end{pmatrix}\>=:\<\rho,\gs\>.$$
We take as our involution 
$$\tau=\tau_{60}:=\begin{pmatrix} 2+\varphi &1-i \\1+i&-2-\varphi \end{pmatrix}.$$
For the prime ideal $\fp=\left(7+5\varphi\right)$ one has $\F_\fp\cong\Q_{59}$,
and the generated group $\Gamma=\left\langle C_{60},\tau_{60}\right\rangle $
is the full ($7+5\vf$)-arithmetic group of $\cO$. 
As such, we establish:
\begin{theorem}
Subject to standard number-theoretic heuristic conjectures, there exists a factorization of any  $g\in\PU(2)$ to precision $\ve$ using $\tau$-count at most $(\nicefrac{7}{3}+o(1))\log_{59}(\nicefrac{1}{\ve^3}).$
\end{theorem}
In particular, if $g$ is near to a diagonal matrix then we just apply the approach in \secref{sec:Diagonal} (for additive error) to get a path of length at most $(1+o(1))\log_{59}(\nicefrac{1}{\ve^3})$, and otherwise run the approach in \secref{sec:AlgShortPath}. 

Notice that no other choice of golden gates is both super and has a greater logarithm base, so that these are the ``best'' generators given the present state of knowledge.

%% file: FactoringGamma.tex
In this section, we establish a key technical lemma regarding approximations in the matrix group, and the method for exact synthesis for elements of the relevant subgroup. 

For $\alpha,\beta\in\mathbb{C}$ satisfying $\abs{\alpha}^2+\abs{\beta}^2=1$ we write
$$u(\alpha,\beta):=\begin{pmatrix}\alpha & \beta \\ -\bar{\beta} & \bar{\alpha}\end{pmatrix},$$ and 
for $\theta\in\mathbb{R}$ we write
$$u(\theta):=u(e^{i\theta},0)=\begin{pmatrix}e^{i\theta} \\ & e^{-i\theta}\end{pmatrix}.$$

We restate below \cite[Lemma 5]{Stier} which holds independent of our choice of universal gate set. 
\begin{lemma}\label{lem:old lem}
    Select absolute constants $\delta,\varepsilon_0>0$ and put $C=\sqrt{\half+\half\left(\frac{2+\delta}{\varepsilon_0}\right)^2}$. Take $\gamma_1$ and $\gamma_2$ in either $\SU(2)$ or $\PU(2)$ and write them as $\gamma_1=u(\alpha_1,\beta_1)$ and $\gamma_2=u(\alpha_2,\beta_2)$. If $\abs{\abs{\alpha_1}-\abs{\alpha_2}}<\varepsilon$ for some $\varepsilon<\delta$ and $\min\{\abs{\alpha_1},\abs{\alpha_2}\}<\sqrt{1-\varepsilon_0^2}$ then for 
    \begin{align*}
        \theta_1&=\frac{1}{2}(\arg\alpha_1-\arg\alpha_2+\arg\beta_1-\arg\beta_2), & \delta_1&=u(\theta_1)\\
        \theta_2&=\frac{1}{2}(\arg\alpha_1-\arg\alpha_2-\arg\beta_1+\arg\beta_2), & \delta_2&=u(\theta_2)
    \end{align*}
    we have the approximation $\delta_1\gamma_2\delta_2$ to $\gamma_1$, satisfying
    $$d(\gamma_1,\delta_1\gamma_2\delta_2)<C\varepsilon.$$
\end{lemma}
Loosely, this result can be thought of as a sufficient condition to transform one $2\times2$ unitary into (approximately) another based only on a weak condition and by ``tuning'' with diagonal (rotation) matrices on either side. The condition is merely that the ``starting'' and ``target'' matrices have top-left entries nearby in absolute value, and that neither is very near to 1. 

We now move to factorizing. Put $\eta=7+5\vf$ for the sequel. Observe that is is a positive real number, and the generator of $\fp$ above. For our purposes, we will encounter elements of $\Ga=\<\rho,\gs,\tau\>\le\PU(2)$ only as $\Z[\vf]$-quaternions with norm a power of $\eta$, envisioned as elements of the Cayley graph for $\<\rho,\gs,\tau\>\le\Q(\vf)\U_2(\Z[\vf])$, which \cite{ParzanchevskiSarnak} shows acts transitively with respect to the distance measure of $\tau$-count, which can be detected by counting the power of $\eta$ in the quaternion norm, after quotienting out $\Z[\vf]$-scalars. This leads to the following factoring algorithm. 
\begin{algorithm}\label{alg:factor in G}
    Let $C$ be a set of representatives of $\<\rho,\gs\>\le\PU(2)$ lifted to $\Q(\vf)\U_2(\Z[\vf])$. Given $\ga$ a lift of some element of $\Ga$ with $\tau$-count $k$, it can be factored by determining which (unique) $c\in C$ gives rise to $\ga c\tau$ of $\tau$-count $k-1$, which in turn requires no more than 60 multiplications of three matrices. $c$ is found if and only if $\ga c\tau$ has all coefficients divisble by $\eta$ in $\Z[\vf]$. 
\end{algorithm}
This algorithm has the base case of simply comparing $\ga$ against all 60 elements of $C$, another trivial computational task. The idea of the general case is that $\ga$ has the representation
$$c_0\prod\limits_{i\in[n]}\tau c_i$$
where $c_i$ is not the identity for $i\neq0,n$. Then $c$ will be projectively equal to $c_n\inv$, giving rise to
$$\ga c\tau=c_0\prod\limits_{i\in[n-1]}\tau c_i\tau c_nc\tau=zpc_0\prod\limits_{i\in[n-1]}\tau c_i$$
for $z=\det(c_nc)$ (coprime to $\eta$).

%% file: Diagonal.tex
\subsection{Algorithm}
We proceed by, for given diagonal element $\gd=u(\gt)$ and $\ve$, seeking $\g\in\G$ with $d(\gd,\g)<\ve$. Knowing that $\g$ is projectively equal to an element 
$$\begin{pmatrix}\phantom{-}x_0+x_1i&x_2+x_3i\\-x_2+x_3i&x_0-x_1i\end{pmatrix}$$
for $x_0,x_1,x_2,x_3\in\Z[\vf]$ satisfying 
\begin{equation}
    x_0^2+x_1^2+x_2^2+x_3^2=\eta^m \label{eq:sum of 4 squares diag}
\end{equation}
for some $m\in\N$, we also have that it is sufficient to satisfy
\begin{equation}
    x_0\cos\gt+x_1\sin\gt\ge \eta^{\nicefrac{m}{2}}(1-2\ve^2) \label{eq:dot product diag}
\end{equation}
(note that this is precisely \cite[(13) and Problem 7.4]{RossSelinger} (see there for the derivation)), following the readily computable observation that $\norm{\gd-\g}\ge\sqrt{2}d(\gd,\g)$ and setting the goal $\norm{\gd-\g}\le\sqrt{2}\ve$, since \cite{RossSelinger} operates with respect to the operator norm. Observe, unfortunately, that \eqref{eq:sum of 4 squares diag} is a quadratic constraint and \eqref{eq:dot product diag} is a linear constraint. We now explain how to transform them into a practicable sequence of integer programs.\footnote{Fundamentally, this is the same idea as in \cite{RossSelinger}, as their study of upright ellipses accomplishes the same task as Lenstra's algorithm.}

Fix $m$. We seek $x_\ell\in\Z[\vf]$ satisfying \eqref{eq:sum of 4 squares diag} and \eqref{eq:dot product diag}. Define $y_\ell=\nicefrac{x_\ell}{\eta^{\nicefrac{m}{2}}}$. Artificially add the constraint 
\begin{equation}
    1-\ve^2-y_1\sin\gt\ge0. \label{eq:y1 artificial constraint}
\end{equation}
Observe from \eqref{eq:sum of 4 squares diag} that $y_0^2\le1-y_1^2$. Then, we have the sequence of implications (mainly by algebraic manipulation)
\begin{align}
    y_0\cos\gt&\ge1-\ve^2-y_1\sin\gt\ge0\nonumber\\
    y_0^2\cos^2\gt&\ge(1-\ve^2)^2+y_1^2\sin^2\gt-2(1-\ve^2)y_1\sin\gt\nonumber\\
    \cos^2\gt&\ge(y_1-(1-\ve^2)\sin\gt)^2-(1-\ve^2)^2\sin^2\gt+(1-\ve^2)^2\nonumber\\
    \cos^2\gt(2\ve^2-\ve^4)&\ge(y_1-(1-\ve^2)\sin\gt)^2\nonumber\\
    \abs{y_1-(1-\ve^2)\sin\gt}&\le\abs{\cos\gt}\sqrt{2-\ve^2}\ve\label{eq:y1 main cons}
\end{align}
and so we have reduced our consideration to just one of the four variables. As $[\Q(\vf):\Q]=2$, we explicitly work with the Galois group elements
\begin{align*}
    \gs_+:\,1&\longmapsto1\\
        \vf&\longmapsto\vf,\\
    \gs_-:\,1&\longmapsto1\\
        \vf&\longmapsto\vf^\bullet,
\end{align*}
where $\vf^\bullet$ is $\vf$'s Galois conjugate $\frac{1-\sqrt{5}}{2}=1-\vf$, both of which are real embeddings, yielding the additional constraints
\begin{equation}
    \abs{\gs_\pm y_1}\le1.\label{eq:y1 triv cons}
\end{equation}
Now, putting $x_1=c+d\vf$ for some $c,d\in\Z$, \eqref{eq:y1 artificial constraint}, \eqref{eq:y1 main cons}, and \eqref{eq:y1 triv cons} transform into 
\begin{problem}\label{prob:x1}
Given $m$, find $(c,d)\in\Z^2$ satisfying the five linear constraints
\begin{align*}
    (c+d\vf)\sin\gt&\le\eta^{\nicefrac{m}{2}}(1-\ve^2),\\
    \abs{c+d\gs_\pm\vf}&\le(\gs_\pm\eta)^{\nicefrac{m}{2}},\\
    \abs{c+d\vf-\eta^{\nicefrac{m}{2}}(1-\ve^2)\sin\gt}&\le\eta^{\nicefrac{m}{2}}\abs{\cos\gt}\sqrt{2-\ve^2}\ve.
\end{align*}
\end{problem}
The existence of such a pair is determinable in $O(\poly m)$ time, by Lenstra's algorithm \cite{Lenstra,Paz}. 

Putting $x_0=a+b\vf$ for some $a,b\in\Z$, for a particular solution to \probref{prob:x1}, \eqref{eq:sum of 4 squares diag} and \eqref{eq:dot product diag} transform into
\setcounter{problem}{-1}
\begin{problem}\label{prob:x0}
Given $m$ and $x_1=c+d\vf$, find $(a,b)\in\Z^2$ satisfying the four linear constraints
\begin{align*}
    \abs{a+b\gs_\pm\vf}&\le(\gs_\pm\eta)^{\nicefrac{m}{2}}\sqrt{1-(\gs_\pm x_1)^2},\\
    (a+b\vf)\cos\gt&\le\eta^{\nicefrac{m}{2}}(1-x_1\sin\gt),\\
    (a+b\vf)\cos\gt&\ge\eta^{\nicefrac{m}{2}}(1-\ve^2-x_1\sin\gt).
\end{align*}
\end{problem}
Again, the existence of such a pair is determinable in $O(\poly m)$ time, by Lenstra's algorithm. 

Finally, if we have solutions to \probref{prob:x0} and \probref{prob:x1}, we seek to solve
\setcounter{problem}{22}
\begin{problem}\label{prob:x2,x3}
Given $m$, $x_0=a+b\vf$, and $x_1=c+d\vf$, find $(x_2,x_3)\in\Z[\vf]^2$ satisfying \eqref{eq:sum of 4 squares diag}. 
\end{problem}
Here we change techniques. The objective is to write $\eta^m-x_0^2-x_1^2$ as a sum of squares in $\Z[\vf]$. Assuming efficient factorization in $\Z$ (or $\Z[\vf]$, also a PID), \probref{prob:x2,x3} is efficiently solvable via \thmref{thm:many SOTS} (the general approach being basically identical to the classical algorithms for $\Z[i]$ or $\Z[e^{\nicefrac{i\pi}{8}}]$, cf. \cite[\S C]{RossSelinger} for the latter). 

We have succeeded in the core of the algorithm. To handle given $\gd$, begin by fixing $m=0$. For given $m$, attempt to solve \probref{prob:x1}; for each solution, attempt to solve \probref{prob:x0}; for each solution, attempt to solve \probref{prob:x2,x3}. If this results in a tuple $(x_0,x_1,x_2,x_3)$ satisfying all three problems, halt and return 
$$\frac{1}{\eta^{\nicefrac{m}{2}}}\begin{pmatrix}\phantom{-}x_0+x_1i&x_2+x_3i\\-x_2+x_3i&x_0-x_1\end{pmatrix}.$$
Otherwise, if all possibilities have been exhausted, increment $m$. 

\subsection{Analysis}
We refer the reader to \cite[\S2.3]{ParzanchevskiSarnak}'s analysis of timing and correctness for Ross and Selinger's algorithm generalized to any golden gate set, where the conclusion is that for desired precision $\ve$, the factorization length becomes $(1+o(1))\log_{59}(\nicefrac{1}{\eps^3})$ with required computational time remaining $O(\poly\log(\nicefrac{1}{\eps}))$. 

%% file: AlgForShortPaths.tex
Here we largely adopt the structure and wording from \cite[\S4 and \S5]{Stier}, as the core concepts and heuristics that make the algorithm work are the same in the icosahedral setting. 

\subsection{Algorithm}
Select absolute constants $\tilde{\ve},\ve_0>0$. Take any $g=u(\alpha,\beta)\in\PU(2)$ where $\abs{\alpha}<\sqrt{1-\ve_0^2}$, and pick $\ve<\tilde{\ve}$. We wish to approximate $g$ in $d$ using $\g\in\G$ of the form
\begin{equation}
    \g=\frac{1}{\eta^{\nicefrac{k}{2}}}\begin{pmatrix}\phantom{-}x_0+x_1i&x_2+x_3i\\-x_2+x_3i&x_0-x_1i\end{pmatrix}\label{eq:g def}
\end{equation}
having $k$, the factorization length, minimized, and so we begin with $k=0$. (We also have $x_0,x_1,x_2,x_3\in\Z[\vf]$.) In particular, the objective is to approximate $g$ as $\g_1\g\g_2$ where $\g_1,\g_2\in\G$ approximate well-chosen diagonals computable using \secref{sec:Diagonal}, and $\g\in\G$ has factorization computable using \algref{alg:factor in G}. We will see that $\g$ is designed to have factorization typically shorter than that of $\g_1$ and $\g_2$, giving rise to the desired improvement. 

In order to apply \lemref{lem:old lem} we need to have $\abs{\frac{x_0+x_1i}{\eta^{\nicefrac{k}{2}}}}=\sqrt{\frac{x_0^2+x_1^2}{\eta^k}}$ near $\abs{\alpha}$ (that is, within $\ve$). Because $\abs{\frac{x_0+x_1i}{\eta^{\nicefrac{k}{2}}}}+\abs{\alpha}\ge\abs{\alpha}$ which is fixed, it suffices to find candidate values for $x_0,x_1\in\Z[\vf]$ with $\abs{\abs{\frac{x_0+x_1i}{\eta^{\nicefrac{k}{2}}}}^2-\abs{\alpha}^2}<\ve\abs{\alpha}$, rewritten to
\begin{equation}
    \abs{x_0^2+x_1^2-\abs{\alpha}^2\eta^k}<\ve\abs{\alpha}\eta^k.\label{eq:linear abs constraint}
\end{equation}
Viewing $\g$ as an element of $\SU(2)$, we also have $\det\g=1$, i.e. $x_0^2+x_1^2+x_2^2+x_3^2=\eta^k$. As $x_\ell\in\Z[\vf]\subset\Q(\vf)\subset\R$, it follows that $\gs_\pm(x_0^2+x_1^2)+\gs_\pm(x_2^2+x_3^2)=\gs_\pm(\eta^k)$, and so 
\begin{equation}
    \gs_\pm(x_0^2+x_1^2)\le (\gs_\pm\eta)^k.\label{eq:galois sq sum bd}
\end{equation}
Now, let $m=x_0^2+x_1^2\in\Z[\vf]$. Considering $\Z[\vf]$ as an integer lattice, we adapt \eqref{eq:linear abs constraint} and \eqref{eq:galois sq sum bd} and seek to solve
\begin{align}
    \abs{m-\abs{\alpha}^2\eta^k}&<\ve\abs{\alpha}\eta^k\label{eq:final lenstra 1}\\
    m&\le\eta^k\label{eq:final lenstra 2}\\
    m^\bullet&\le(\eta^\bullet)^k\label{eq:final lenstra 3}\\
    m,m^\bullet&\ge0\label{eq:final lenstra 4}
\end{align}
which are convex constraints on $m$'s lattice components. Since this is an integer programming problem in two dimensions, we apply Lenstra's algorithm \cite{Lenstra,Paz} to efficiently list all such lattice points $m$. For each $m$, using \thmref{thm:many SOTS}, we attempt to write $m$ as a sum of two squares; if possible, say $m=x_0^2+x_1^2$, we then attempt to write $\tilde{m}=\eta^k-m$ as a sum of two squares. If possible, say $\tilde{m}=x_2^2+x_3^2$, so we simply halt and return $\g$ corresponding to \eqref{eq:g def}. However, if $\tilde{m}$ may not be represented as a sum of two squares, we simply move on to the next value of $m$ and try this process again. If this fails for all $m$ arising from $k$, we increment $k$ and run Lenstra's algorithm for the new inequalities. 

Supposing we have halted and constructed $\g$, we compute $\gd_1$ and $\gd_2$ guaranteed by \lemref{lem:old lem}. These are efficiently approximable by \secref{sec:Diagonal} to $\g_1$ and $\g_2$, respectively. Chaining together the three approximations as $\g_1\g\g_2$ gives the final desired approximation to $g$. 

\subsection{Analysis}
We begin the analysis by establishing the $\tau$-count and tightness of the approximation. In particular, $d(\g_1,\gd_1)<\ve$ and $d(\g_2,\gd_2)<\ve$ with factorization lengths of $\g_\ell$ each up to $(1+o(1))\log_{59}(\nicefrac{1}{\ve^3})$. By \lemref{lem:old lem}, $d(\g,\gd_1\g\gd_2)<C\ve$. Therefore, $d(g,\g_1\g\g_2)<(C+2)\ve$ (by the triangle inequality) and since $\g$ has a factorization of $\tau$-count up to $\lpr{\frac{1}{3}+o(1)}\log_{59}(\nicefrac{1}{\ve^3})$, this constitutes a factorization of an element in $g$'s neighborhood of $\tau$-count up to $\lpr{\frac{7}{3}+o(1)}\log_{59}(\nicefrac{1}{\ve^3})$. 

The efficiency of this algorithm---that is, that it runs in time $O(\poly\log(\nicefrac{1}{\ve}))$---is because we expect to halt when $59^k\ve\in O(1)$ (so only $k\approx\frac{1}{3}\log_{59}(\nicefrac{1}{\ve^3})$ calls are expected), and only call polynomially-many polynomial-time subroutines. The dominant subroutines are calls to Lenstra's algorithm which as shown in \cite{Lenstra} runs in time polynomial in the size of the constraints for any fixed dimension $n$. Indeed, here we have only $n=2$ dimensions, $m=6$ linear constraints (two per absolute value), and the largest value $a$ in the constraints is $\eta^k$, so the runtime is polynomial in $nm\log a\in \Theta(k)$. 

The reason we expect to halt when $59^k\ve\in O(1)$ is that heuristically, we expect to halt when the area enclosed by \eqref{eq:final lenstra 1}--\eqref{eq:final lenstra 4} is $O(\poly\log(\nicefrac{1}{\eps}))$. Conveniently, the region is a rectangle since the vectors $\langle1,\vf\rangle$ and $\langle1,\vf^\bullet\rangle$ are orthogonal, so assuming in the limit that $\ve\ll\min\left\{\abs{\alpha},\frac{1-\abs{\alpha}^2}{\abs{\alpha}}\right\}$ (so that \eqref{eq:final lenstra 2} and the ``bottom'' inequality of \eqref{eq:final lenstra 1} are redundant), we compute length-times-width of
$$\frac{2\ve\abs{\alpha}\eta^k}{\sqrt{1+\vf^2}}\cdot\frac{(\eta^\bullet)^k}{\sqrt{1+(1-\vf)^2}}=\frac{2\abs{\alpha}}{\sqrt{5}}\cdot59^k\ve\in O(\poly\log(\nicefrac{1}{\ve}))$$
whence we find $k\in(1+o(1))\log_{59}(\nicefrac{1}{\ve})$ as expected. 

When attempting to write elements of $\cO$ as a sum of two squares, we primarily rest on a belief, in the style of Cram\'er's conjecture and a conjecture of Sardari \cite[($*$)]{Sardari} that sums of squares are dense in $\N$. Seeking to analogize \cite[($*$)]{Sardari} in particular, we note that the operative aspect is that a dense cluster of lattice points will represent at least one sum of two squares, and that such a point thus will be found quickly through Lenstra's algorithm. 

The significance of this result is to accomplish a factorization in $\PU(2)$ in analogue to \cite{Stier}, but using a gate set with additional desirable properties beyond those enjoyed by the Clifford+$T$ gates.

%% file: Ex.tex
Our algorithm has been implemented for proof-of-concept purposes in Python, and the code is available at \url{https://math.berkeley.edu/~zstier/icosahedral}. Included in this implementation are:
\begin{itemize}
    \item An implementation of Lenstra's algorithm for special cases (\texttt{convex.py}). 
    \item The rings $\Z[\vf]$, $\Z[i,\vf]$, and $H(\Z[\vf])$ (\texttt{rings.py} and \texttt{quaternions.py}). 
    \item Solutions to sum-of-two-square problems in $\Z[\vf]$ (\texttt{rings.py}). 
    \item Factorization of elements of $H(\Z[\vf])$ which are of norm a power of $\eta$ (\texttt{quaternions.py}). 
    \item Efficient factorization of diagonal elements of $\PU(2)$, as outlined in \secref{sec:Diagonal} (\texttt{diagonal.py}). 
    \item Efficient factorization of general elements of $\PU(2)$, as outlined in \secref{sec:AlgShortPath} (\texttt{approx.py}). 
\end{itemize}
In the remainder of this section, we demonstrate the efficacy of this factorization technique on the two more ``classical'' single-qubit quantum gate generators; recall
$$H=\frac{i}{\sqrt{2}}\begin{pmatrix}1&\phantom{-}1\\1&-1\end{pmatrix} \text{ and } T=\begin{pmatrix}e^{\nicefrac{i\pi}{8}}&\\&e^{\nicefrac{-i\pi}{8}}\end{pmatrix}.$$
In both cases, pick precision $\ve=\nicefrac{1}{10^{10}}$. 

For the first example of $T$, we yield 
\begin{align*}
    T &\approx \xi_0\tau\xi_1\tau\xi_2\tau\xi_3\tau\xi_4\tau\xi_5\tau\xi_6\tau\xi_7\tau\xi_8\tau\xi_9\tau\xi_{10}\tau\xi_{11}\tau\xi_{12}\tau\xi_{13}\tau\xi_{14}\tau\xi_{15}\tau\xi_{16}\tau\xi_{17}\tau\xi_{18}\tau\xi_{19} \\
        &=(\sigma\sigma\rho\sigma\rho)\tau(\rho\sigma\sigma\rho\sigma\sigma)\tau(\sigma\rho\sigma\rho\sigma)\tau(\rho\sigma\rho\sigma\rho\sigma)\tau(\rho\sigma\sigma\rho \\
        &\quad\,\sigma\rho\sigma\sigma\rho\sigma\rho)\tau(\sigma\sigma\rho\sigma\rho\sigma\sigma\rho\sigma\rho)\tau(\rho)\tau(\sigma\rho\sigma\sigma\rho\sigma\rho\sigma\sigma\rho) \\
        &\quad\,\tau(\rho)\tau(\rho\sigma\sigma\rho\sigma\rho\sigma\sigma\rho)\tau(\rho\sigma\rho\sigma\rho\sigma\sigma)\tau(\sigma\sigma\rho\sigma\rho\sigma\sigma\rho) \\
        &\quad\,\tau(\sigma\rho\sigma\rho\sigma\sigma)\tau(\sigma\rho\sigma\sigma\rho\sigma\rho\sigma\sigma\rho)\tau(\sigma\rho\sigma\rho\sigma\sigma\rho\sigma)\tau(\sigma\rho \\
        &\quad\,\sigma\rho\sigma)\tau(\rho\sigma\sigma\rho\sigma\rho\sigma\sigma\rho\sigma\rho)\tau(\rho\sigma\sigma\rho\sigma\rho\sigma\sigma\rho\sigma)\tau(\sigma\rho\sigma)
\end{align*}
where $\xi_i\in\<\gs,\rho\>$ correspond to the parenthesized terms following; this achieves precision in $d$ of
$\nicefrac{1.28}{10^{10}}$. The $\tau$-count is 
19; compare this with the predicted $\log_{59}(\nicefrac{1}{\ve^3})\approx16.9$. We remark that the discrepancy can be attributed to computational limitations: at key steps in the algorithm, we must run the Tonelli--Shanks algorithm for finding quadratic residues modulo some prime $q$, which has worst-case behavior $\Theta(q)$. Therefore it is only practical to abandon on instances where some prime factor is at least $10^6$, something that occurred more than 328 times before the above-stated approximation was found. 

For the second example of $H$, the central element is determined to be the following: 
\begin{align*}
    \gamma &= \xi_0\tau\xi_1\tau\xi_2\tau\xi_3\tau\xi_4\tau\xi_5\tau\xi_6\tau\xi_7\tau\xi_8\tau\xi_9 \\
        &=(\rho\sigma\sigma\rho\sigma\rho\sigma\sigma\rho\sigma\rho)\tau(\rho\sigma\rho\sigma\rho\sigma\sigma\rho\sigma\sigma)\tau(\sigma\rho\sigma\rho\sigma\sigma)\tau \\
        &\quad\,(\rho\sigma\sigma\rho\sigma\rho\sigma\sigma\rho)\tau(\rho\sigma\sigma\rho)\tau(\sigma\rho\sigma\sigma\rho\sigma\rho)\tau(\sigma\rho\sigma\rho\sigma\sigma\rho \\
        &\quad\,\sigma\sigma)\tau(\rho\sigma\rho\sigma\sigma\rho\sigma\rho\sigma\sigma\rho\sigma)\tau(\rho\sigma\sigma\rho\sigma\rho\sigma)\tau(\rho\sigma\sigma\rho\sigma\rho)
\end{align*}
($\xi_i$ serves the same function as above). It has $\tau$-count 9; compare this with the predicted $\frac{1}{3}\log_{59}(\nicefrac{1}{\ve^3})\approx5.6$. As before, the discrepancy can be attributed to non-vanishing of $o(1)$ factors for explicit $\ve$ and to having to abandon possibilities with very large prime factors. 
The outer diagonal elements are determined to be the following: 
\begin{align*}
    \gamma_1 &\approx \tsup[1]{\xi}_0\tau\tsup[1]{\xi}_1\tau\tsup[1]{\xi}_2\tau\tsup[1]{\xi}_3\tau\tsup[1]{\xi}_4\tau\tsup[1]{\xi}_5\tau\tsup[1]{\xi}_6\tau\tsup[1]{\xi}_7\tau\tsup[1]{\xi}_8\tau\tsup[1]{\xi}_9\tau\tsup[1]{\xi}_{10}\tau\tsup[1]{\xi}_{11}\tau\tsup[1]{\xi}_{12}\tau\tsup[1]{\xi}_{13}\tau\tsup[1]{\xi}_{14}\tau\tsup[1]{\xi}_{15}\tau\tsup[1]{\xi}_{16}\tau\tsup[1]{\xi}_{17}\tau\tsup[1]{\xi}_{18} \\
        &=(\sigma\rho\sigma\rho)\tau(\rho\sigma\rho\sigma\sigma\rho\sigma\rho\sigma\sigma\rho\sigma)\tau(\rho\sigma\sigma\rho\sigma\sigma)\tau(\rho\sigma\sigma\rho\sigma \\
        &\quad\,\rho\sigma\sigma\rho\sigma\rho)\tau(\rho\sigma\rho\sigma\rho\sigma)\tau(\sigma\rho\sigma\rho)\tau(\sigma\sigma\rho\sigma)\tau(\sigma\sigma\rho\sigma\rho\sigma) \\
        &\quad\,\tau(\sigma\rho\sigma\rho\sigma\sigma\rho\sigma)\tau(\sigma\rho\sigma\rho\sigma\sigma\rho\sigma\rho)\tau(\rho\sigma\sigma\rho\sigma\rho\sigma)\tau(\rho\sigma \\
        &\quad\,\rho\sigma\sigma\rho)\tau(\sigma\sigma\rho\sigma)\tau(\sigma\sigma\rho\sigma\rho\sigma\sigma)\tau(\sigma\sigma)\tau(\sigma\rho\sigma\rho\sigma\sigma\rho\sigma\sigma) \\
        &\quad\,\tau(\sigma\rho\sigma\rho\sigma\sigma\rho\sigma)\tau(\sigma\sigma\rho\sigma\rho\sigma\sigma\rho\sigma\rho)\tau(\rho\sigma\rho\sigma\sigma\rho\sigma\rho\sigma) \\
    \gamma_2 &\approx \hat{\xi}_0\tau\hat{\xi}_1\tau\hat{\xi}_2\tau\hat{\xi}_3\tau\hat{\xi}_4\tau\hat{\xi}_5\tau\hat{\xi}_6\tau\hat{\xi}_7\tau\hat{\xi}_8\tau\hat{\xi}_9\tau\hat{\xi}_{10}\tau\hat{\xi}_{11}\tau\hat{\xi}_{12}\tau\hat{\xi}_{13}\tau\hat{\xi}_{14}\tau\hat{\xi}_{15}\tau\hat{\xi}_{16}\tau\hat{\xi}_{17}\tau\hat{\xi}_{18} \\
        &=(\sigma\rho\sigma\rho\sigma\sigma\rho)\tau(\sigma\sigma\rho\sigma\rho\sigma\sigma\rho)\tau(\sigma\rho\sigma\rho\sigma\sigma\rho\sigma\sigma)\tau(\sigma\rho\sigma \\
        &\quad\,\sigma\rho\sigma\rho)\tau(\sigma\sigma\rho\sigma\rho\sigma)\tau(\rho\sigma\rho\sigma\sigma)\tau(\sigma\rho\sigma\sigma\rho\sigma\rho)\tau(\rho\sigma\rho\sigma \\
        &\quad\,\rho\sigma\sigma\rho\sigma\rho)\tau(\rho\sigma\rho\sigma\sigma\rho\sigma\sigma)\tau(\rho\sigma\sigma\rho\sigma)\tau(\rho\sigma\rho\sigma\rho\sigma\sigma\rho \\
        &\quad\,\sigma\rho)\tau(\rho\sigma\sigma\rho)\tau(\rho\sigma\sigma\rho\sigma\rho\sigma\sigma\rho)\tau(\rho\sigma\sigma\rho\sigma)\tau(\rho\sigma\rho\sigma\sigma\rho \\
        &\quad\,\sigma\rho\sigma\sigma)\tau(\rho\sigma\rho\sigma\sigma\rho\sigma\rho\sigma\sigma)\tau(\sigma\sigma\rho\sigma\rho\sigma\sigma\rho\sigma)\tau(\sigma\rho)\tau(\rho)
\end{align*}
($\tsup[1]{\xi}_i$, $\hat{\xi}_i$ serves the same function as above). They both have $\tau$-counts of 18, with 75 collective abandoned cases; compare this with the predicted $\log_{59}(\nicefrac{1}{\ve^3})\approx16.9$. Multiplying out $\g_1\g\g_2$ gives distance in $d$ of $\nicefrac{1.28}{10^{10}}$ (that this difference equals the previous one is a coincidence; they disagree at the third decimal place). 

%% file: ConcludingRemarks.tex
This work represents a theoretical, heuristic, and proof-of-concept demonstration of a state-of-the-art methodology to construct single-qubit quantum gates, optimizing for using as few expensive gates as possible, and in particular representing an improvement of $\log_259\approx5.9$ over the previous best method \cite{Stier}. However, the area of efficient quantum hardware selection is far from fully explored. While \cite{ParzanchevskiSarnak} demonstrates that efficiently computing a length-$\log_p(\nicefrac{1}{\eps^3})$ factorization is NP-complete (for an arbitrary $\PU(2)$-element into golden gates associated to prime $p$), it may still be possible to achieve length-$c\log_p(\nicefrac{1}{\eps^3})$ factorizations for $c\in(1,\nicefrac{7}{3})$. Another possibility is in the study of multiple qubits simultaneously, as has been initiated for $\PU(3)$ by Evra and Parzanchevski \cite{EvraParzanchevski}. 

%% file: Acknowledgements.tex
TRB was supported, in part, by the Institute for Advanced Study and Medgar Evers College. ZS was supported by a National Science Foundation graduate research fellowship, grant numbers DGE-1752814/2146752.
We are deeply indebted to Peter Sarnak for introducing us to this to this problem and for the many conversations that led to production of this paper. We are also grateful to Carlos Esparza and John Voight for their suggestions. 

%% file: norm-Euclidean.tex
Put $R=\Z[i,\vf]$ and $K=\Q(i,\vf)$. 
\begin{proposition}
    $R$ is the ring of integers of $K$. 
\end{proposition}
\begin{proof}
    Consider the canonical $\Z$-basis of $R$, namely $\{1,\vf,i,i\vf\}$. One readily computes its discriminant to be 400, equal to $K$'s, as per \cite{lmfdb}. 
\end{proof}

Consider the norm function $N=\abs{N_{K/\Q}}$, defined on $K$ and taking integral values on $R$. (The absolute value here is customary but superfluous, as one of $K$'s $\Q$-automorphisms is complex conjugation.)

In what follows, balls $\cB(r)$ of radius $r$ are centered at the origin and closed and with respect to the $\ell_\infty$-norm. 

We shall treat $K$ interchangably with its formulation as the $\Q$-space $\Q^4$. 

\begin{theorem}\label{thm:norm-euclidean}
    $R$ is norm-Euclidean; that is, $R$ is a Euclidean domain with respect to the Euclidean function $N$. 
\end{theorem}
\begin{proof}
    We use the standard reformulation that $R$ is norm-Euclidean if and only if for all $\alpha\in K$ there is $\gb\in R$ for which $N(\alpha-\gb)<1$. Therefore we shall attack the latter statement. 
    
    Put $\alpha=w+x\vf+yi+zi\vf\in K$. Then, we readily compute
    \begin{align*}
        N(\alpha) &= w^4 + 2 w^3 x - w^2 x^2 - 2 w x^3 + x^4 + 2 w^2 y^2 + 2 w x y^2 + 3 x^2 y^2 + y^4 + 2 w^2 y z \\
            &\qquad - 8 w x y z - 2 x^2 y z + 2 y^3 z + 3 w^2 z^2 - 2 w x z^2 + 2 x^2 z^2 - y^2 z^2 - 2 y z^3 + z^4.
    \end{align*}
    Define $\norm{\alpha}=\max\{\abs{w},\abs{x},\abs{y},\abs{z}\}$. Then we can also compute, for fixed $\alpha$ and any other $\gb\in K$ with $\norm{\gb}\to0$, 
    \begin{equation}
        N(\alpha+\gb)\le N(\alpha)+O(\norm{\gb}),\label{eq:small perturbation}
    \end{equation}
    and we shall presently obtain an effective, non-asymptotic form of this fact. 
    
    Viewing $K\cong\Q^4$ as $\Q$-spaces and $R$ as the standard lattice, we can translate any element of $K$ using $R$ to one with $\ell_\infty$-norm at most $\nicefrac{1}{2}$ (that is, $\cB(\nicefrac{1}{2})$), by subtracting off the ``rounded'' element---round each component to the nearest integer. So, it remains to verify whether every element $\alpha\in\cB(\nicefrac{1}{2})$ of the vector space has $N(\alpha)<1$ (which turns out to not actually hold!). To attempt to accomplish this, we look at a refinement of $R$. In particular, consider $\Lambda=\frac{1}{n}\Z^4$ for well-chosen $n$. We cover $\cB(\nicefrac{1}{2})$ with $\Lambda$-translates of $\cB(\nicefrac{1}{2n})$, so that for each $\alpha\in\gL$, for all $\beta\in\alpha+\cB(\nicefrac{1}{2n})$, by \eqref{eq:small perturbation} we have $N(\beta)\le N(\alpha)+O(\nicefrac{1}{2n})$. The balancing act, then, is to choose $n$ small enough so that $\#(\Lambda\cap\cB(\nicefrac{1}{2}))$ is manageable for a computer, but large enough so that $\nicefrac{1}{2n}$ is sufficiently small. There is a catch, where it is not actually the case that we can ensure that $N(\alpha)+O(\nicefrac{1}{2n})<1$ for all $\alpha\in\Lambda\cap\cB(\nicefrac{1}{2})$; however, for those $\alpha$ which violate this condition, we can attempt to circumvent the obstruction by translating to $(\alpha+\delta)+\cB(\nicefrac{1}{2n})$ (where $\delta\in\Z^4$) and then taking $N$.\footnote{As it turns out, picking $\delta$ so that it shifts exactly one component towards 0 by exactly 1 is sufficient when any $\delta$ is necessary at all.}
    
    We now establish \eqref{eq:small perturbation}, after which we will be able to select $n$. Keeping $\alpha=w+x\vf+yi+zi\vf$, put $\beta=d_1+d_2\vf+d_3i+d_4\vf$. Fully written out, $N(\alpha+\beta)$ is (the absolute value of) a degree-four polynomial in eight variables with 170 total terms, so we spare the reader its full statement. Letting $\norm{\beta}\le\ve$,\footnote{So that $\abs{d_\ell}\le\ve$ for all $\ell\in[4]$.} however, we have that
    \begin{align*}
    N(\alpha+\beta) &\le N(\alpha) \\ 
        & \quad + 2\ve (\abs{2 w^3 + 3 w^2 x - w x^2 - x^3 + 2 w y^2 + x y^2 + 2 w y z + 3 w z^2 - x z^2 - 4 x y z} \\
            & \qquad + \abs{w^3 - w^2 x - 3 w x^2 + 2 x^3 + w y^2 + 3 x y^2 - 4 w y z - 2 x y z - w z^2 + 2 x z^2} \\
            & \qquad + \abs{2 w^2 y + 2 w x y + 3 x^2 y + 2 y^3 + w^2 z - 4 w x z - x^2 z + 3 y^2 z - y z^2 - z^3} \\
            & \qquad + \abs{w^2 y - 4 w x y - x^2 y + y^3 + 3 w^2 z - 2 w x z + 2 x^2 z - y^2 z - 3 y z^2 + 2 z^3}) \\
        & \quad + \ve^2 (\abs{6 w^2 + 6 w x - x^2 + 2 y^2 + 2 y z + 3 z^2} + 2\abs{3 w^2 - 2 w x - 3 x^2 + y^2 - 4 y z - z^2} \\
            & \qquad + \abs{ - w^2 - 6 w x + 6 x^2 + 3 y^2 - 2 y z + 2 z^2} + \abs{2 w^2 + 2 w x + 3 x^2 + 6 y^2 + 6 y z - z^2} \\
            & \qquad + 2\abs{w^2 - 4 w x - x^2 + 3 y^2 - 2 y z - 3 z^2} + \abs{3 w^2 - 2 w x + 2 x^2 - 6 y z + 6 z^2 - y^2} \\
            & \qquad + 4\abs{2 w y + x y + w z - 2 x z} + 4\abs{w y + 3 x y - 2 w z - x z} + 4\abs{w y - 2 x y + 3 w z - x z} \\
            & \qquad + 4\abs{ - 2 w y - x y - w z + 2 x z}) \\
        & \quad + 2\ve^3 (\abs{w + x} + 2\abs{3w - x} + 2\abs{w + 3x} + 2\abs{2x - w} + 3\abs{2w + x} + 2\abs{w - 2x} \\
            & \qquad + 2\abs{2z - y} + 2\abs{y + 3z} + 2\abs{y - 2z} + 2\abs{3y - z} + 4\abs{2y + z}) \\
        & \quad + 40 \ve^4.
    \end{align*}
    Then, picking $n=6$ and $\ve=\frac{1}{12}$, for each $\alpha\in\Lambda\cap\cB(\nicefrac{1}{2})$ as well as $\alpha-(\sgn w,0,0,0)$, $\alpha-(0,\sgn x,0,0)$, $\alpha-(0,0,\sgn y,0)$, and $\alpha-(0,0,0,\sgn z)$, we test whether the right-hand side of the above is bounded by 1 for any of those five choices. The code appearing in \figref{fig:norm code} verifies that this indeed comes to pass. 
\end{proof}
\begin{figure}
\centering
\begin{lstlisting}[language=Python]
import numpy as np
def KQnorm(w,x,y,z,r):
    p0 = abs(w**4 + 2 * w**3 * x - w**2 * x**2 - 2 * w * x**3 + x**4 \ 
        + 2 * w**2 * y**2 + 2 * w * x * y**2 + 3 * x**2 * y**2 \ 
        + y**4 + 2 * w**2 * y * z - 8 * w * x * y * z \ 
        - 2 * x**2 * y * z + 2 * y**3 * z + 3 * w**2 * z**2 \ 
        - y**2 * z**2 - 2 * y * z**3 + z**4 - 2 * w * x * z**2 \ 
        + 2 * x**2 * z**2)
    p1 = 2 * r * (abs(2 * w**3 + 3 * w**2 * x - w * x**2 - x**3 \ 
        + 2 * w * y**2 + x * y**2 + 2 * w * y * z + 3 * w * z**2 \ 
        - x * z**2 - 4 * x * y * z) + abs(w**3 - w**2 * x \ 
        - 3 * w * x**2 + 2 * x**3 + w * y**2 + 3 * x * y**2 \ 
        - 4 * w * y * z - 2 * x * y * z - w * z**2 + 2 * x * z**2) \ 
        + abs(2 * w**2 * y + 2 * w * x * y + 3 * x**2 * y + 2 * y**3 \ 
        + w**2 * z - 4 * w * x * z - x**2 * z + 3 * y**2 * z \ 
        - y * z**2 - z**3) + abs(w**2 * y - 4 * w * x * y - x**2 * y \ 
        + y**3 + 3 * w**2 * z - 2 * w * x * z + 2 * x**2 * z \ 
        - y**2 * z - 3 * y * z**2 + 2 * z**3))
    p2 = r**2 * (abs(6 * w**2 + 6 * w * x - x**2 + 2 * y**2 \ 
        + 2 * y * z + 3 * z**2) + abs(6 * w**2 - 4 * w * x \ 
        - 6 * x**2 + 2 * y**2 - 8 * y * z - 2 * z**2) \ 
        + abs( - w**2 - 6 * w * x + 6 * x**2 + 3 * y**2 - 2 * y * z \ 
        + 2 * z**2) + abs(2 * w**2 + 2 * w * x + 3 * x**2 + 6 * y**2 \ 
        + 6 * y * z - z**2) + abs(2 * w**2 - 8 * w * x - 2 * x**2 \ 
        + 6 * y**2 - 4 * y * z - 6 * z**2) + abs(3 * w**2 \ 
        - 2 * w * x + 2 * x**2 - 6 * y * z + 6 * z**2 - y**2) \ 
        + abs(8 * w * y + 4 * x * y + 4 * w * z - 8 * x * z) \ 
        + abs(4 * w * y + 12 * x * y - 8 * w * z - 4 * x * z) \ 
        + abs(4 * w * y - 8 * x * y + 12 * w * z - 4 * x * z) \ 
        + abs( - 8 * w * y - 4 * x * y - 4 * w * z + 8 * x * z)) 
    p3 = 2 * r**3 * (abs(w + x) + 2 * abs(3 * w - x) \ 
        + 2 * abs(w + 3 * x) + 2 * abs(2 * x - w) \ 
        + 3 * abs(2 * w + x) + 2 * abs(w - 2 * x) \ 
        + 2 * abs(2 * z - y) + 2 * abs(y + 3 * z) \ 
        + 2 * abs(y - 2 * z) + 2 * abs(3 * y - z) \ 
        + 4 * abs(2 * y + z)) 
    p4 = 40 * r**4
    return p0 + p1 + p2 + p3 + p4
n = 7
r = 1.0/(2*(n-1))
l = np.linspace(-0.5,0.5,n)
for a in l:
	for b in l:
		for c in l:
			for d in l:
				if KQnorm(a,b,c,d,r) >= 1 \ 
				    and KQnorm(a-np.sign(a),b,c,d,r) >= 1 \ 
				    and KQnorm(a,b-np.sign(b),c,d,r) >= 1 \ 
				    and KQnorm(a,b,c-np.sign(c),d,r) >= 1 \ 
				    and KQnorm(a,b,c,d-np.sign(d),r) >= 1:
					print (a,b,c,d)
\end{lstlisting}
\caption{Code used to prove \thmref{thm:norm-euclidean}. It prints 4-tuples corresponding to points whose norms are too large. Its failure to print anything completes the proof.}
\label{fig:norm code}
\end{figure}
\begin{remark}
It would appear that this method would readily adapt to a computation to determine that additional number fields are norm-Euclidean, so long as one knows a basis for its ring of integers and correspondingly computes an appropriate analogue to the local effective bound \texttt{KQnorm(w,x,y,z,r)}. Unfortunately we have been unable to reproduce this success with any other biquadratic number field of the form $\Q\lpr{i,\sqrt{n}}$, $n>5$. 
\end{remark}

%% file: Zphi.tex
As in \cite[\S C]{RossSelinger}, here we shall summarize results about classifying irreducible elements in $\Z[\vf]$ and an efficient algorithm for writing certain elements of $\Z[\vf]$ as a sum of two squares. 

For this section, let $N=N_{\Q(\vf)/\Q}$ (in contrast to $N=N_{\Q(i,\vf)/\Q}$ as in \secref{sec:norm-Euclidean}). We readily compute $N(a+b\vf)=a^2+ab-b^2$. 

Recall that $\Z[\vf]$ is a Euclidean domain with respect to $\abs{N}$, and that $\Z[\vf]^\times=\<\pm\vf\>$ (cf. \cite[\S4.1.4]{ParzanchevskiSarnak}). Also recall that $\vf^\bullet$ is $\vf$'s Galois conjugate, which happens to equal $1-\vf$. Extend $^\bullet$ $\Q$-linearly to all of $\Q(\vf)$. 

\begin{proposition}\label{prop:primes in Z vf}
    Let $p$ be irreducible in $\Z$. If $p\equiv\pm2\pmod{5}$ then $p$ is irreducible in $\Z[\vf]$. If $p\equiv\pm1\pmod{5}$ then there is an algorithm (running in time $O(\poly\log p)$) to compute a $\Z[\vf]$-irreducible element dividing $p$. 
\end{proposition}
\begin{proof}
    If $p$ is reducible in $\Z[\vf]$ then there exist non-units $u,v\in\Z[\vf]$ with $uv=p$ and accordingly $N(u)\mid N(p)=p^2$; that is, $N(u)=p$. 
    \begin{itemize}
    \item Observe that $a^2+ab-b^2\equiv(a-2b)^2\pmod{5}$ and that neither of $\pm2$ are quadratic residues modulo 5. This proves the first part of the proposition, as if $u=a+b\vf$ then $N(u)$ can never equal $p$. 
    
    \item To prove the second part, we first show that if $p\equiv1\pmod{5}$ then there exists $x\in\Z$ satisfying $x^2-x-1\equiv0\pmod{p}$. Indeed, rewrite the equation to $(x-2\inv)^2\equiv1+4\inv$. Now, 
    $$\leg{1+4\inv}{p}=\leg{1+4\inv}{p}\leg{4}{p}=\leg{5}{p}=\leg{p}{5}(-1)^{p-1}=\leg{p}{5}=1$$
    with the third equality following by quadratic reciprocity. Let $y\in\Z$ be any modulo-$p$ square root of $1+4\inv$. Then $x=y+2\inv$ suffices. Set $z=x-\vf\in\Z[\vf]$. Then $p\mid x^2-x-1=zz^\bullet$ but $p\nmid z$ as $z$'s $\vf$-part is not divisible by $p$, so $p$ is not irreducible. Thus $u=\gcd(p,z)$ will be a non-unit divisor of $p$, as desired. ($u$ cannot be a unit because otherwise $\gcd(p^\bullet,z^\bullet)=\gcd(p,z^\bullet)$ will also be a unit, but then there are no prime divisors of $zz^\bullet$ also dividing $p$, a contradiction.)
    
    The timing of the above method is due to the extended Euclidean algorithm for modular inverses, and the standard Euclidean algorithm for GCDs in a Euclidean domain, which have respective runtimes $O(\log^2p)$ and $O(\log p)$. 
    \end{itemize}
\end{proof}
The only case not covered in the above is $p=5$, which trivially factors as $(-1+2\vf)^2$. 
\begin{algorithm}\label{alg:factor quickly}
    Assume that there is an efficient blackbox algorithm to factor $n\in\Z$ in time $O(\poly\log n)$. Then there is an algorithm which factors $n\in\Z[\vf]$ in time $O(\poly\log\abs{N(n)})$. 
\end{algorithm}
\begin{proof}
    Factor $N(n)$ as $\prod\limits_{\ell=1}^mp_\ell^{e_\ell}$. Each prime $p_\ell$ is factorizable in time $O(\poly\log p_\ell)\subset O(\poly\log N(n))$, by \propref{prop:primes in Z vf}. There are $O(\log N(n))$ such primes. 
\end{proof}
To each irreducible $u\in\Z[\vf]$, let its {\em associated prime} be the least (positive) prime factor of $N(u)$; that is, $\sqrt{N(u)}$ when $u$ is a $\Z$-irreducible times a unit, and $N(u)$ otherwise.\footnote{Let us quickly establish why the associated prime is well-defined. Suppose $p\neq q$ are $\Z$-primes with $p,q\mid N(u)=uu^\bullet$. Let $p',q'\in\Z[\vf]$ be irreducibles such that $\abs{N(p')}=p$ if $p\equiv\pm1\pmod{5}$ or $p=5$, and $p'=p$ otherwise; and define  $q'$ similarly. Then $p',q'\mid uu^\bullet$, so as irreducibles we have $p'$ dividing either $u$ or $u^\bullet$, and similarly for $q'$, hence $p=q$ since $u$ and $u^\bullet$ are themselves irreducible. Thus, $N(u)$ is divisible by at most one prime at most twice.}
\begin{lemma}\label{lem:not two SOTS}
    For any $u\in\Z[\vf]\backslash\{0\}$, it is not the case that both $u$ and $u\vf$ can be written as a sum of two squares. 
\end{lemma}
\begin{proof}
    Suppose not, so write $u=w^2+x^2$ and $u\vf=y^2+z^2$. Then we may take the product 
    $$u^\bullet u\vf=N(u)\vf=(w^\bullet y+x^\bullet z)^2+(w^\bullet z-x^\bullet y)^2=(a+b\vf)^2+(c+d\vf)^2$$
    for some $a,b,c,d\in\Z$. The right-hand side expands out with 1-part\footnote{Viewing $\Q(\vf)\cong\<1,\vf\>_\Q\cong\Q^2$ so that the ``1-part'' and ``$\vf$-part'' are the corresponding components in the canonical isomorphism.} equal to $a^2+b^2+c^2+d^2$, but $N(u)\vf$ has 1-part equal to 0, so $a=b=c=d=0$, and therefore $N(u)=0$, a contradiction. 
\end{proof}

\begin{proposition}\label{prop:irreds SOTS}
    Let $u\in\Z[\vf]$ be irreducible, with associated prime $p$. If $p\equiv1,3,7,9,13,17\pmod{20}$ then there is an efficient algorithm in time $O(\poly\log p)$ to write either $u$ or $u\vf$ as a sum of two squares. 
\end{proposition}
Observe that $\phi(20)=8$ ($\phi$ here signifying Euler's totient function), so by Chebotarev's density theorem the above-asserted algorithm manages to write $\nicefrac{3}{4}$ of primes as sums of two squares. 

\begin{proof}
    Before beginning in earnest, first remark that 2 factors in $\Z[i,\vf]$ as $(1+i)(1-i)$, and that $1\pm i$ are irreducible because $N_{\Q(i,\vf)/\Q(\vf)}(1\pm i)=2$, a $\Z[\vf]$-irreducible. 
    
    \begin{itemize}
    \item Suppose $p\equiv1\pmod{4}$ (i.e. $p\equiv1,9,13,17\pmod{20}$). Then there exists even $x\in\Z$ satisfying $x^2+1\equiv0\pmod{p}$, that is, $u\mid p\mid x^2+1$. Let $g=\gcd(u,x+i)$, computed in $\Z[i,\vf]$, and set $s=\re(g)$ and $t=\im(g)$. We claim that either $u$ or $u\vf$ equals a squared unit times $s^2+t^2$. Indeed, suppose $q\in\Z[i,\vf]$ is an irreducible divisor of $u$. We wish to show that $q$ divides exactly one of $x\pm i$. (Clearly it divides at least one of them, by $q\mid u\mid x^2+1$.)
    \begin{itemize}
    \item Say $1\pm i\neq q\in\Z[i,\vf]$ is an irreducible dividing $u$ and both $x\pm i$. Then $q\mid 2i$. By our choice of $q$ this is impossible. 
    
    \item Suppose now we let $q$ be one of $1\pm i$. Then in order to have $1\pm_1i\mid x\pm_2i$ (where $\pm_1$ and $\pm_2$ are independent signs) we must have $(1\pm_1i)(a+bi)=x\pm_2i$ where $a,b\in\Z[\vf]$. Solving for $a$ and $b$ by looking at real and imaginary parts, we have
    \begin{align*}
        a\mp_1 b&=\phantom{\pm_2}x\\
        b\pm_1 a&=\pm_21
    \end{align*}
    and solving yields $2a=x\pm_1\pm_21$, so in order for $a$ to exist we must have $x$ is odd. This is a contradiction to our assumption that $x$ is even, regardless of the choice of $\pm_1$ and $\pm_2$, thus in fact we have neither of $1\pm i$ dividing either of $x\pm i$. 
    \end{itemize}
    Thus, if $q$ divides $u$ with multiplicity $m_q$ (so that $u$ factors in $\Z[i,\vf]$ as $u=\vf^m\prod\limits_{q\mid u}q^{m_q}$), we must have that $q$ divides $x\pm i$ with multiplicity at least $m$ and $x\mp i$ with multiplicity 0; and so\footnote{Depending on choice of GCD algorithm, there may be a unit factor in front, but we assume not without loss of generality.} $g=\prod\limits_{q\mid u,x+i}q^{m_q}$ while $\bar{g}=\gcd(u,x-i)=\prod\limits_{q\mid u,x-i}q^{m_q}$. Thus $u=\vf^mg\bar{g}=\vf^m(\re(g)^2+\im(g)^2)$. If $2\mid m$ then $u=(\vf^{\nicefrac{m}{2}}\re(g))^2+(\vf^{\nicefrac{m}{2}}\im(g))^2$. 

    \item Suppose $p\equiv3\pmod{4},\pm2\pmod{5}$ (i.e. $p\equiv3,7\pmod{20}$). Then compute
    $$\leg{-5}{p}=\leg{-1}{p}\leg{5}{p}=-\leg{p}{5}(-1)^{p-1}=1$$
    so that there exists even $x\in\Z$ which is not a multiple of 5 satisfying $x^2+5\equiv0\pmod{p}$, that is, $u=p\mid x^2+5$. Let $g=\gcd(g,x+i(-1+2\vf))$, computed in $\Z[i,\vf]$, and set $s=\re(g)$ and $t=\im(g)$. We claim that $p=s^2+t^2$. Indeed, suppose $q\in\Z[i,\vf]$ is an irreducible divisor of $p$ (in $\Z[i,\vf]$). We wish to show that $q$ divides exactly one of $x\pm i(-1+2\vf)$. (Clearly it divides at least one of them, by $q\mid p\mid x^2+5$.)
    \begin{itemize}
        \item Say $-1+2\vf,1\pm i\neq q\in\Z[i,\vf]$ is an irreducible dividing $u$ and both $x\pm i(-1+2\vf)$. Then $q\mid2(-1+2\vf)i$. By our choice of $q$ this is impossible. 
        \item Suppose we let $q=-1+2\vf=\sqrt{5}$. Then as $5\nmid x$, we have 
        $$\frac{x\pm i(-1+2\vf)}{q}=\frac{x}{5}(-1+2\vf)\pm i\in\Q(i,\vf)\backslash\Z[i,\vf]$$
        and so in fact $q$ divides neither of $x\pm i(-1+2\vf)$. 
        \item Suppose now we let $q$ be one of $1\pm i$. Then in order to have $1\pm_1i\mid x\pm_2i(-1+2\vf)$ (where $\pm_1$ and $\pm_2$ are independent signs) we must have $(1\pm_1i)(a+bi)=x\pm_2i(-1+2\vf)$ where $a,b\in\Z[\vf]$. Solving for $a$ and $b$ by looking at real and imaginary parts, we have 
        \begin{align*}
            a\mp_1b&=\hspace{1.7cm}x\\
            b\pm_1a&=\pm_2(-1+2\vf)
        \end{align*}
        and solving yields $2a=x\pm_1\pm_2(-1+2\vf)$, so in order for $a$ to exist we must have $x$ is odd. This is a contradiction to our assumption that $x$ is even, regardless of the choice of $\pm_1$ and $\pm_2$, thus in fact we have neither of $1\pm i$ dividing either of $x\pm i$. 
    \end{itemize}
    \end{itemize}

    It is here that we require \thmref{thm:norm-euclidean}, that $\Z[i,\vf]$ is norm-Euclidean, as this ensures that GCDs are efficiently computable, using an Euclidean algorithm with respect to the Euclidean function (the norm). As before, the bottlenecks are in the extended and standard Euclidean algorithms, which are both of runtime $O(\log^2p)$. 
\end{proof}
Observe that we can also include the $\Z$-irreducibles $2=1^2+1^2$ and $5=(-1+2\vf)^2+0^2$. 

\begin{theorem}\label{thm:many SOTS}
    For $x\in\Z[\vf]$, factor it as 
    $$x=\prod\limits_{\ell=1}^nu_\ell^{m_\ell}$$
    where $u_\ell\in\Z[\vf]$ are all irreducible, with associated primes $p_\ell$, and $m_\ell\in\N$. Then either $x$ or $x\vf$ may be written as the sum of two squares if, for all $\ell\in[n]$, one of the following holds: 
    \begin{itemize}
        \item $p_\ell=2$; 
        \item $p_\ell\equiv1,3,7,9,13,17\pmod{20}$; 
        \item $2\mid m_\ell$; 
    \end{itemize}
    and, for arbitrary $x$ not a priori satisfying all three criteria, this whole procedure (i.e., either determining a sum of two squares, or rejecting on the basis of some $\ell$ failing all three conditions) runs in time $O(\poly\log\abs{N(x)})$. 
\end{theorem}
\begin{proof}
    For each $\ell\in[n]$, consider $s_\ell$ and $t_\ell$ defined as follows: 
    \begin{itemize}
        \item If $p_\ell=2$ then $s_\ell=t_\ell=1$. 
        \item If $p_\ell\equiv1,3,7,9,13,17\pmod{20}$ then $s_\ell$ and $t_\ell$ are as computed using \propref{prop:irreds SOTS}. 
        \item If $2\mid m_\ell$ then $s_\ell=u_\ell^{\nicefrac{m_\ell}{2}}$ and $t_\ell=0$. 
    \end{itemize}
Now, we do the standard (inductive) trick where $(s_\ell^2+t_\ell^2)(s_{\ell'}^2+t_{\ell'}^2)=(s_\ell s_{\ell'}+t_\ell t_{\ell'})^2+(s_\ell t_{\ell'}-t_\ell s_{\ell'})^2$. Since we have that $s_\ell^2+t_\ell^2$ equals either $u_\ell$ or $u_\ell\vf$, the resulting product from performing the trick $n$ times gives a sum of two squares $s^2+t^2$ equal to $x\vf^{n'}$ where $0\le n'\le n$. If $2\mid n'$ then we return the pair $(s\vf^{\nicefrac{-n'}{2}},t\vf^{\nicefrac{-n'}{2}})$, and otherwise we return $(s\vf^{\nicefrac{(1-n')}{2}},t\vf^{\nicefrac{(1-n')}{2}})$. 

For the runtime analysis, simply factor $x$ using \algref{alg:factor quickly} and then for each resulting factor, apply \propref{prop:irreds SOTS} after checking that one of the three criteria holds (immediately halting and rejecting if any factor fails). 
\end{proof}